\documentclass[letterpaper,11pt,oneside]{amsart}
\usepackage{amsmath}
\usepackage{amssymb}
\newcommand{\urltilde}{\~{}}                          
\newcommand{\revauthor}[1]{\relax}

\usepackage{amsrefs}


\newtheorem{theorem}{Theorem}[section]
\newtheorem{lemma}[theorem]{Lemma}

\theoremstyle{definition} 

\newtheorem{definition}[theorem]{Definition} 
\newtheorem{remark}[theorem]{Remark}

\newcommand{\define}[1]{\textit{#1}}

\newcommand{\RCAo}{\mathsf{RCA}_0}

\newcommand{\rcaw}{{\sf {RCA}}^\omega}
\newcommand{\rca}{{\sf{RCA}}}

\newcommand{\haw}{{\widehat{\mathsf{E}\text{-}\mathsf{HA}}}\strut^\omega_{\raise2pt\hbox{\scriptsize${\mathord{\upharpoonright}}$}}}

\newcommand{\hatwo}{\widehat{\mathsf{E}\text{-}\mathsf {HA}}\strut^2_{\raise2pt\hbox{\scriptsize${\mathord{\upharpoonright}}$}}}

\newcommand{\hawfi}{\mathsf{E}\text{-}\mathsf {HA}^\omega}
\newcommand{\hafitwo}{\mathsf{E}\text{-}\mathsf {HA}^2}

\newcommand{\whaw}{{\widehat{\mathsf {WE}\text{-}\mathsf {HA}}} ^\omega_{\raise2pt\hbox{\scriptsize${\mathord{\upharpoonright}}$}}}

\newcommand{\whawfi}{\mathsf{W}\hawfi}

\newcommand{\ia}{\mathsf{IA}}
\newcommand{\ac}{\mathsf{AC}}
\newcommand{\ipwef}{\mathsf{IP}^{\mspace{2mu}\omega}_{\text{\textup{ef}}}}
\newcommand{\ipwa}{\mathsf{IP}^{\mspace{1mu}\omega}_{\forall}}
\newcommand{\markov}{\mathsf{M}^\omega}

\newcommand{\REC}{\text{\textrm{REC}}}

\newcommand{\mr}{\operatorname{\mathsf{mr}}}
\newcommand{\lang}{{\mathcal L}}

\newcommand{\setN}{\mathbb{N}} 
\newcommand{\setQ}{\mathbb{Q}}
\newcommand{\seq}[1]{#1}
\newcommand{\seqx}{\langle x_n \mid {n \in \setN}\rangle}
\newcommand{\seqy}{\langle y_n \mid {n \in \setN} \rangle}

\usepackage{hyperref}

\begin{document}

\title[Reverse mathematics and uniformity in proofs]{Reverse mathematics and uniformity in proofs without excluded middle}

\author{Jeffry L. Hirst}
\address{Department of Mathematical Sciences\\
Appalachian State University\newline
Boone, NC  28608, USA}
\email{jlh@math.appstate.edu}
\urladdr{www.mathsci.appstate.edu/\urltilde jlh}

\author{Carl Mummert}
\address{Department of Mathematics\\
Marshall University\newline
One John Marshall Drive\\
Huntington,  WV 25755, USA}
\email{mummertc@marshall.edu}
\urladdr{www.science.marshall.edu/mummertc}

\date{\today}

\begin{abstract}
  We show that when certain statements are provable
  in subsystems of constructive analysis using intuitionistic
  predicate calculus, related sequential statements are provable in
  weak classical subsystems.  In particular, if a $\Pi^1_2$ sentence
  of a certain form is provable using
  E-HA${}^\omega$ along with the axiom of
  choice and an independence of premise principle, the sequential form
  of the statement is provable in the classical system RCA.
  We obtain this and similar results using applications of modified
  realizability and the \textit{Dialectica} interpretation. These
  results allow us to use techniques of classical reverse mathematics
  to demonstrate the unprovability of several mathematical principles in
  subsystems of constructive analysis.
\end{abstract}

\maketitle

\section{Introduction}

We study the relationship between systems of intuitionistic arithmetic
in all finite types (without the law of the excluded middle) and weak
subsystems of classical second order arithmetic.  Our theorems give
precise expressions of the informal idea that if a sentence $\forall
X\, \exists Y\, \Phi(X,Y)$ is provable without the law of the excluded
middle, then the proof should be sufficiently direct that the stronger
\textit{sequential form}
\[
\forall \langle X_n \mid n \in \setN \rangle\,
\exists \langle Y_n \mid n \in \setN\rangle\, \forall n\,
 \Phi(X_n,Y_n)
 \] 
is provable in a weak subsystem of classical arithmetic. We call
our theorems ``uniformization results'' because the provability of
the sequential form demonstrates a kind of uniformity in the proof of
the original sentence.

The subsystems of classical arithmetic of interest are 
$\rca_0$, which is well-known in Reverse 
Mathematics~\cite{Simpson-SOSOA}, and its extension $\rca$ with 
additional induction axioms. These systems are closely related to 
computable analysis. In particular, both subsystems are satisfied in the 
model $\REC$ that has the set $\omega$ of standard natural numbers as 
its first order part and the collection of all computable subsets of 
$\omega$ as its second order part. When the conclusions of our 
uniformization results are viewed as statements about $\REC$, they 
provide a link between constructive analysis and computable analysis.  
Moreover, because $\rca_0$ is the base system most often employed in 
Reverse Mathematics, our results also provide a link between the 
fields of Reverse Mathematics and constructive analysis. Full 
definitions of the subsystems of intuitionistic and classical arithmetic 
that we study are presented in section~\ref{sec2}.

In section~\ref{sec3}, we prove uniformization results using modified
realizability, a well-known tool in proof theory. In particular, we
show there is a system $I_0$ of intuitionistic arithmetic in all finite types 
such that whenever an $\forall\exists$ statement of a certain
syntactic form is provable in $I_0$, its sequential form is
provable in $\rca_0$ (Theorem~\ref{719J}). Moreover, the system $I_0$ contains the full scheme for the
axiom of choice in all finite types, which is classically much
stronger than $\rca_0$. We have attempted to make section~\ref{sec3}
accessible to a general reader who is familiar with mathematical logic
but possibly unfamiliar with modified realizability.

In section~\ref{sec4}, we give several examples of theorems in
classical mathematics that are provable in $\rca_0$ but not provable
in~$I_0$. These examples demonstrate empirically that the syntactic
restrictions within our uniformization theorems are not excessively
tight. Moreover, our uniformization theorems allow us to obtain these 
unprovability results simply by showing that the sequential versions of the
statements are unprovable in $\rca_0$, which can be done using
classical techniques common in Reverse Mathematics.  In this way, we
obtain results on unprovability in intuitionistic arithmetic solely through a
combination of our uniformization theorems and the study of classical
arithmetic.  A reader who is willing to accept the results of
section~\ref{sec3} should be able to skim that section and then
proceed directly to section~\ref{sec4}.

In section~\ref{sec5}, we prove uniformization results for $\rca_0$
and $\rca$ using the {\it Dialectica} interpretation of G\"odel.
These results allow us to add a Markov principle to the system of
intuitionistic arithmetic in exchange for shrinking the class of
formulas to which the theorems apply.

We would like to thank Jeremy Avigad and Paulo Oliva for
helpful comments on these results.  We began this work
during a summer school on proof theory taught by Jeremy
Avigad and Henry Towsner at Notre Dame in 2005.  Ulrich
Kohlenbach generously provided some pivotal insight during
the workshop on Computability, Reverse Mathematics, and
Combinatorics at the Banff International Research Station in~2008, 
and much additional assistance in later conversations.

\section{Axiom systems}\label{sec2}

Our results make use of subsystems of intuitionistic and classical
arithmetic in all finite types. The definitions of these systems rely on 
the standard type notation in which the type of a natural number is
$0$ and the type of a function from objects of type $\rho$ to objects
of type $\tau$ is $\rho \to \tau$.  For example, the type of a
function from numbers to numbers is $0 \to 0$.  As is typical in the
literature, we will use the types $1$ and $0\to 0$ interchangeably,
essentially identifying sets with their characteristic functions.  We
will often write superscripts on quantified variables to indicate
their type.
%

Full definitions 
of the following systems are given by Kohlenbach~\cite{Koh-book}*{section~3.4}.

\begin{definition}
The system $\haw$ is a theory of intuitionistic arithmetic in all finite types first defined by Feferman~\cite{Feferman-1977}.

The language $\lang(\haw)$ includes the constant 0; the successor,
addition, and multiplication operations; terms for primitive recursion
on variables of type $0$; and the projection and substitution
combinators (often denoted $\Pi_{\rho,\tau}$ and
$\Sigma_{\delta,\rho,\tau}$ \cite{Koh-book}) which allow terms to be
defined using $\lambda$ abstraction.  For example, given $x \in \setN$
and an argument list $t$, $\haw$ includes a term for $\lambda t.x$,
the constant function with value $x$.

The language includes equality as a primitive relation only for type
$0$ objects (natural numbers). Equality for higher types is defined
pointwise in terms of equality of lower types, using the following
extensionality scheme
\[ {\mathsf{E}}\colon \forall x^\rho \forall y^\rho \forall z^{\rho\to
    \tau}\, ( x =_\rho y \to z(x) = _\tau z(y) ).
\]

The axioms of $\haw$ consist of this extensionality scheme, the
basic arithmetical axioms, the defining axioms for the term-forming
operators, and an axiom scheme for induction on quantifier-free
formulas (which may have parameters of arbitrary types).
\end{definition}

\begin{definition}[Troelstra~\cite{troelstra73}*{1.6.12}]
  The subsystem $\hawfi$ is an extension of $\haw$ with additional terms
 and stronger
  induction axioms.  Its language contains additional term-forming
  recursors $R_\sigma$ for all types $\sigma$.  Its new axioms include
  the definitions of these recursors and the full induction scheme
  \[\ia\colon A(0) \to (\forall n(A(n) \to A(n+1)) \to \forall n
  A(n)),\]
in which $A$ may have parameters of arbitrary types.
\end{definition}

The following class of formulas will have an important role in our results. These are, informally, the formulas that have no existential commitments in intuitionistic systems.
\begin{definition}
  A formula of $\lang(\haw)$ 
is \textit{$\exists$-free} if it is built from prime (that is, atomic)
formulas using only universal quantification and the connectives
$\land$ and $\to$. Here the symbol $\bot$ is treated as a prime
formula, and a negated formula $\lnot A$ is treated as an abbreviation for
$A \to \bot$; thus $\exists$-free formulas may include both $\bot$ and~$\lnot$.
\end{definition}

We will consider extensions of $\haw$ and $\hawfi$ that include additonal axiom schemes. The following schemes have been discussed by Kohlenbach~\cite{Koh-book} 
and by Troelstra~\cite{troelstra73}.

\begin{definition} The following axiom schemes are defined in
  $\lang(\hawfi)$.  When we adjoin a scheme to $\haw$, we implicitly
  restrict it to $\lang(\haw)$. The formulas in these schemes may have
parameters of arbitrary types.
\begin{list}{$\bullet$}{}

\item  \textit{Axiom of Choice}.  For any $x$ and $y$ of finite type,
\[ \ac \colon \forall x\, \exists y A(x,y) \to \exists Y\, \forall x\, A(x,Y(x)).\]

\item \textit{Independence of premise for $\exists$-free formulas}.
  For $x$ of any finite type, if $A$ is $\exists$-free and
  does not contain $x$, then
\[\ipwef\colon (A \to \exists x B(x)) \to \exists x (A \to B(x)).\]

\item \textit{Independence of premise for universal formulas}. If
  $A_0$ is quantifier free, $\forall x$ represents a block
  of universal quantifiers, and $y$ is of any type and is
  not free in $\forall x A_0(x)$, then
\[\ipwa\colon (\forall x A_0(x) \to \exists y B(y)) \to \exists y (\forall x A_0 (x) \to B(y)).\]

\item \textit{Markov principle for quantifier-free formulas}. If $A_0$ is quantifier-free and
  $\exists x$ represents a block of existential quantifiers
  in any finite type, then
\[\markov\colon \neg\neg \exists x A_0 (x) \to \exists x A_0 (x).\]
\end{list}
 
\end{definition}

\subsection{Classical subsystems}

The full scheme $\ac$ for the axiom of choice in all finite types, 
which is commonly included in subsystems of intuitionistic arithmetic, 
becomes extremely strong in the presence of the law of the excluded
middle. For this reason, we will be interested in the restricted choice
scheme
\[
{\mathsf{QF}{\text-}\mathsf{AC}}^{\rho,\tau}\colon \forall x^\rho\, \exists y^\tau A_0(x,y) \to \exists Y^{\rho \to \tau}\, \forall x^\rho A_0 (x, Y(x) ),
\]
where $A_0$ is a quantifier-free formula that may have parameters.

We obtain subsystems of classical arithmetic by adjoining forms of this 
scheme, along with the law of the excluded middle, to systems
of intuitionistic arithmetic.  Because these systems include the law
of the excluded middle, they also include all of classical predicate
calculus.
\begin{definition} The system $\rcaw_0$ consists of $\haw$ plus ${\sf
      QF{\text-}AC}^{1,0}$ and the law of the excluded middle.

 The system $\rcaw$ consists of $\hawfi$ (which includes full
    induction) plus ${\sf QF{\text-}AC}^{1,0}$ and the law of the
    excluded middle.  
\end{definition}

We are also interested in the following second order restrictions of
these subsystems. Let $\hatwo$ represent the restriction of $\haw$ to
formulas in which all variables are type $0$ or $1$, and let
$\hafitwo$ be the similar restriction of $\hawfi$ in which variables
are limited to types $0$ and $1$ and the recursor constants are
limited to those of type~$0$.

\begin{definition}
 The system $\rca_0$ consists of $\hatwo$ plus ${\sf
      QF{\text-}AC}^{0,0}$ and the law of the excluded middle.

 The system $\rca$ consists of $\hafitwo$ (which includes the full
    induction scheme for formulas in its language) plus ${\sf
      QF{\text-}AC}^{0,0}$ and the law of the excluded middle.  
\end{definition}

The system $\rca_0$ (and hence also $\rcaw_0$) is able to prove the
induction scheme for $\Sigma^0_1$ formulas using ${\sf QF}\text{-}{\sf
  AC}^{0,0}$ and primitive recursion on variables of type~$0$,
 as noted by Kohlenbach~\cite{Koh-HORM}.

The following conservation results show that the second order subsystems 
$\rca$ and $\rca_0$ have the same deductive strength for
sentences in their restricted languages as the corresponding
higher-type systems $\rca^\omega$ and $\rca^\omega_0$,
respectively.

\begin{theorem}\label{consrcao}{\cite{Koh-HORM}*{Proposition~3.1}}
  For every sentence $\Phi$ in $\lang(\rca_0)$, 
  if $\rcaw_0 \vdash \Phi$ then $\rca_0 \vdash \Phi$.
\end{theorem}

The proof of this theorem is
based on a formalization of the extensional model of the hereditarily
continuous functionals ($\mathsf{ECF}$), as presented in section~2.6.5
of Troelstra~\cite{troelstra73}.  The central notion is that
continuous objects of higher type can be encoded by lower type objects.
For example, if $\alpha$ is a functional of type $1 \to 0$ and
$\alpha$ is continuous in the sense that the value of $\alpha (X)$
depends only on a finite initial segment of the characteristic
function of $X$, then there is an {\sl associated function}
\cite{Kleene} of type $0 \to 0$ that encodes all the information
needed to calculate values of~$\alpha$.  Generalizing this notion,
with each higher-type formula $\Phi$ we can associate a second order
formula $\Phi_{\sf ECF}$ that encodes the same information.  The proof
sketch for the following result indicates how this is applied to
obtain conservation results.

\begin{theorem}\label{consrca}
  For each sentence $\Phi$\/ in $\lang(\rca)$, if\/ $\rcaw \vdash \Phi$
  then $\rca \vdash \Phi$.
\end{theorem}
\begin{proof}
  The proof proceeds in two steps.  First, emulating section~2.6.5 and
  Theorem~2.6.10 of Troelstra~\cite{troelstra73}, show that if $\rcaw
  \vdash \Phi$ then $\rca \vdash \Phi_{\sf ECF}$.  Second, following
  Theorem~2.6.12 of Troelstra \cite{troelstra73}, prove that if $\Phi$
  is in the language of $\rca$ then $\rca \vdash \Phi \leftrightarrow
  \Phi_{\sf ECF}$.
\end{proof}
  
The classical axiomatization of $\RCAo$, presented by
Simpson~\cite{Simpson-SOSOA}, uses the set-based language $L_2$ with
the membership relation symbol~$\in$, rather than the language based
on function application used in~$\haw$.  The systems defined above as
$\RCAo$ is sometimes denoted $\RCAo^2$ to indicate it is a restriction
of $\RCAo^\omega$.  As discussed by Kohlenbach~\cite{Koh-HORM},
set-based $\RCAo$ and function-based $\RCAo^2$ are each included in a
canonical definitional extension of the other, and the same holds for
set-based $\rca$ and function-based $\rca^2$.  Throughout this paper,
we use the functional variants of $\RCAo$ and $\rca$ for convenience,
knowing that our results apply equally to the traditionally
axiomatized systems.

\section{Modified realizability}\label{sec3}

Our most broadly applicable uniformization theorems are 
proved by an application of modified realizability, a
technique introduced by Kreisel~\cite{KR}.  Excellent expositions on
modified realizability are given by
Kohlenbach~\cite{Koh-book} and
Troelstra~\cites{troelstra73,troelstra-HP}.  Indeed, our proofs 
make use of only minute modifications of results stated in these sources.

Modified realizability is a scheme for matching each formula $A$ with
a formula \mbox{$t \mr A$} with the intended meaning ``the sequence
of terms $t$ realizes~$A$.'' 

\begin{definition}\label{719A}
Let $A$ be a formula in $\lang ( \hawfi )$, and let $\seq x$ denote a possibly
empty tuple of terms whose variables do not appear free in $A$.  The formula
$\seq x \mr A$ is defined inductively as follows:
\begin{list}{}{}
\item [(1)] $\seq x \mr A$ is $A$, if $\seq x$ is empty and $A$ is a prime formula.
\item [(2)] $\seq x , \seq y \mr (A \land B)$ is $\seq x \mr A \land \seq y \mr B$.
\item [(3)] $z^0, \seq x, \seq y \mr (A \lor B)$ is $(z = 0 \to \seq x \mr A) \land (z \neq 0 \to \seq y \mr B)$.
\item [(4)] $\seq x \mr (A \to B)$ is $\forall y ( \seq y \mr A \to \seq x \seq y \mr B)$.
\item [(5)] $\seq x \mr (\forall y^\rho A(y))$ is $\forall y^\rho (\seq x \seq y \mr A(y))$.
\item [(6)] $z^\rho , \seq x \mr (\exists y^\rho A(y))$ is $\seq x \mr A(z)$.
\end{list}
Note that if $A$ is a prime formula then $A$ and $t \mr A$ are identical; this is even true for $\exists$-free formulas if we ignore dummy quantifiers.
\end{definition}

We prove each of our uniformization results in two steps. The
first step shows that whenever an $\forall \exists$ statement is
provable in a particular subsystem of intuitionistic arithmetic, we
can find a sequence of terms that realize the statement.  The second
step shows that a classical subsystem is able to leverage the terms in
the realizer to prove the sequential version of the original
statement.

We begin with systems containing the full induction scheme.
For the first step, we require the following theorem.

\begin{theorem}[\cite{Koh-book}*{Theorem~5.8}]\label{719B}
Let $A$ be a formula in $\lang ( \hawfi )$.  If\/
\[
\hawfi +\ac + \ipwef \vdash A
\]
then there is a tuple $t$ of terms of $\lang(\hawfi)$ 
such that $\hawfi  \vdash t \mr A$.
\end{theorem}

%

For any formula $A$, $\hawfi + \ac + \ipwef$ is able to prove $A
\leftrightarrow \exists x ( x \mr A)$.  However, the deduction of $A$
from $(t \mr A)$ directly in $\hawfi$ is only possible for some
formulas. 

\begin{definition}\label{719C}
$\Gamma_1$ is the collection of formulas in $\lang ( \hawfi )$ defined inductively as follows.
\begin{list}{}{}
\item [(1)]  All prime formulas are elements of $\Gamma_1$.
\item [(2)] If $A$ and $B$ are in $\Gamma_1$, then so are $A \land B$,
  $A\lor B$, $\forall x A$, and $\exists x A$.
\item [(3)] If $A$ is 
$\exists$-free  and $B$ is in $\Gamma_1$, then $(\exists x
  A \to B)$ is in $\Gamma_1$, where $\exists x$ may represent a block of
  existential quantifiers.
\end{list}
\end{definition}
The class $\Gamma_1$ is sometimes defined in terms of ``negative''
formulas~\cite{troelstra73}*{Definition~3.6.3}, those which can be constructed
from negated prime formulas by means of $\forall$, $\land$, $\to$, and
$\bot$. In all the systems studied in this paper, every $\exists$-free
formula is equivalent to the negative formula obtained by replacing
each prime formula with its double negation. Thus the distinction
between negative and $\exists$-free will not be significant.


The next lemma is proved by Kohlenbach~\cite{Koh-book}*{Lemma~5.20} 
and by Troelstra~\cite{troelstra73}*{Lemma~3.6.5}

\begin{lemma}
\label{719D}
For every formula $A$ in $\lang(\hawfi)$, if $A$ is in $\Gamma_1$, 
then $\hawfi \vdash (t \mr A ) \to A$.
\end{lemma}


Applying Theorem~\ref{719B} and Lemma~\ref{719D}, we now prove the
following term extraction lemma, which is similar to the main
theorem on term extraction via modified realizability (Theorem 5.13) of
Kohlenbach~\cite{Koh-book}.  Note that $\forall x\, \exists y\, A$ is in
$\Gamma_1$ if and only if $A$ is in~$\Gamma_1$.

\begin{lemma}\label{719E}
  Let $\forall x^\rho\, \exists y^\tau A(x,y)$ be a sentence of $\lang(\hawfi)$
  in $\Gamma_1$, where $\rho$ and $\tau$ are arbitrary types.
  If
  \[
\hawfi + \ac + \ipwef \vdash \forall x^\rho\, \exists y^\tau A(x,y),
  \]
  then $\rcaw \vdash \forall x^\rho A(x, t(x))$, where $t$
  is a suitable term of\/~$\lang(\hawfi)$.
\end{lemma}

\begin{proof}
  Assume that  $\hawfi + \ac + \ipwef \vdash \forall x^\rho
  \exists y^\tau A(x,y)$ where $A(x,y)$ is in $\Gamma_1$.
  By Theorem~\ref{719B}, there is a tuple $t$ of terms of 
  $\lang(\hawfi )$ such that $\hawfi$ proves $t \mr \forall x^\rho
  \exists y^\tau A (x,y)$.  By clause (5) of
  Definition~\ref{719A}, $\hawfi  \vdash \forall x^\rho (
  t(x) \mr \exists y^\tau A (x,y))$.  By clause (6) of
  Definition~\ref{719A}, $t$ has the form $t_0 , t_1$
  and $\hawfi \vdash \forall x^\rho [t_1 (x) \mr A (
  x, t_0 (x))]$.  Because $A(x,y)$ is in~$\Gamma_1$, 
  Lemma~\ref{719D} shows that $\hawfi  \vdash \forall x^\rho A ( x, t_0 (x))$.
  Because $\rcaw $ is an extension of $\hawfi$, we see that 
  $\rcaw \vdash \forall x^\rho A (x, t_0 (x))$.
\end{proof}

We are now prepared to prove our first uniformization theorem.

\begin{theorem}\label{719F}
Let $\forall x \exists y A(x,y)$ be a sentence of $\lang(\hawfi)$ in $\Gamma_1$.  If
\[
\hawfi + \ac + \ipwef \vdash \forall x\, \exists y\, A(x,y),
\]
then 
\[
\rcaw \vdash \forall \seqx \, \exists \seqy \, \forall n\,  A(x_n,y_n).
\]
Furthermore, if $x$ and $y$ are both type $1$
\textup{(}set\textup{)} variables, and the formula $\forall x\, \exists y
A(x,y)$ is in $\lang ( \rca )$, then $\rcaw$ may be replaced
by $\rca$ in the implication.
\end{theorem}

\begin{proof}
  Assume that $\hawfi + \ac + \ipwef \vdash \forall x^\rho
  \exists y^\tau A(x,y)$. We may apply Lemma~\ref{719E} to
  extract the term $t$ such that $\rcaw \vdash \forall
  x^\rho A(x, t(x))$.  Working in $\rcaw$, fix any sequence
  $\seqx$.  This sequence is a function
  of type $0 \to \rho$, so by $\lambda$ abstraction we can
  construct a function of type $0 \to \tau$ defined by
  $\lambda n . t(x_n )$. Taking $\seqy$ to be this sequence, we
  obtain $\forall n\, A(x_n , y_n )$.  The final sentence
  of the theorem follows immediately from the fact that $\rcaw$ is
  a conservative extension of $\rca$ for formulas in $\lang
  (\rca )$.
\end{proof}

We now turn to a variation of Theorem~\ref{719F} that replaces
$\hawfi$ and $\rcaw$ with $\haw$ and $\rcaw_0$, respectively.
Lemmas~\ref{719G} and~\ref{719H} are proved by imitating the proofs of
Theorem~\ref{719B} and Lemma~\ref{719D}, respectively, as described in the
first paragraph of section 5.2 of Kohlenbach~\cite{Koh-book}.

\begin{lemma}\label{719G}
Let $A$ be a formula in $\lang ( \haw )$.  If\/
$
\haw +\ac + \ipwef \vdash A
$,
then there is a tuple $t$ of terms of $\lang(\haw)$ such that $\haw \vdash t \mr A$.
\end{lemma}

\begin{lemma}\label{719H}
Let $A$ be a formula of $\lang(\haw)$. If $A$ is in $\Gamma_1$,
 then $\haw \vdash (t \mr A ) \to A$.
\end{lemma}

\begin{lemma}\label{719I}
  Let $\forall x^\rho\, \exists y^\tau A(x,y)$ be a sentence of $\lang(\haw)$
  in $\Gamma_1$, where $\rho$ and $\tau$ are arbitrary types.
  If
  \[
\haw + \ac + \ipwef \vdash \forall x^\rho \exists y^\tau A(x,y),
  \]
  then $\rcaw_0 \vdash \forall x^\rho A(x, t(x))$, where $t$
  is a suitable term of\/~$\lang(\haw)$.
\end{lemma}
\begin{proof}
  Imitate the proof of Lemma~\ref{719E}, substituting Lemma~\ref{719G}
  for Theorem~\ref{719B} and Lemma~\ref{719H} for Lemma~\ref{719D}.
\end{proof}

We now obtain our second uniformization theorem. This is the theorem discussed in the
introduction, where $I_0$ refers to the theory $\haw + \ac + \ipwef$.

\begin{theorem}\label{719J}
Let $\forall x\, \exists y A(x,y)$ be a sentence of $\lang(\haw)$ in $\Gamma_1$.  If
\[
\haw + \ac + \ipwef \vdash \forall x \exists y\, A(x,y),
\]
then 
\[
\rcaw_0 \vdash \forall \seqx \, \exists \seqy \, \forall n\,  A(x_n,y_n).
\]
Furthermore, if $x$ and $y$ are both type $1$
\textup{(}set\textup{)} variables, and the formula $\forall x\, \exists y
A(x,y)$ is in $\lang ( \rca_0 )$, then $\rcaw_0$ may be replaced
by $\rca_0$ in the implication.
\end{theorem}

The proof is parallel to that of Theorem~\ref{719F}, which did not make
use of induction or recursors on higher types. Theorem~\ref{consrcao} 
serves as the conservation result to prove the final claim.

\section{Unprovability results}\label{sec4}

We now demonstrate several theorems of core mathematics which are
provable in $\RCAo$ but have sequential versions that are not provable
in $\rca$.  In light of Theorem~\ref{719F}, such theorems are not
provable in $\hawfi + \ac + \ipwef$. Where possible, we carry out
proofs using restricted induction, as this gives additional
information on the proof-theoretic strength of the principles being
studied.  The terminology in the following theorem is well known; we
give formal definitions as needed later in the section.

\begin{theorem}\label{thm1}
  Each of the following statements is provable in\/ $\RCAo$
  but not provable in\/ $\hawfi + \ac + \ipwef$.
  \begin{enumerate}
  \item Every $2 \times 2$ matrix has a Jordan
    decomposition.
  \item Every quickly converging Cauchy sequence of
    rational numbers can be converted to a Dedekind cut
    representing the same real number.
  \item Every enumerated filter on a countable poset can be
    extended to an unbounded enumerated filter.
  \end{enumerate}%
\end{theorem}
There are  many
other statements that are provable in $\RCAo$ but not
$\hawfi + \ac + \ipwef$; we have chosen these three to
illustrate the what we believe to be the ubiquity of this
phenomenon in various branches of core mathematics.  

We will show that each of the statements (\ref{thm1}.1)--(\ref{thm1}.3) 
is unprovable in $\hawfi + \ac + \ipwef$ by noting that each statement is
in $\Gamma_1$ and showing that the sequential form of each statement
implies a strong comprehension axiom over $\RCAo$. Because these
strong comprehension axioms are not provable even with the added
induction strength of $\rca$, we may apply Theorem~\ref{719F} to
obtain the desired results.  The stronger comprehension axioms include
weak K\"onig's lemma and the arithmetical comprehension scheme, which
are discussed thoroughly by Simpson~\cite{Simpson-SOSOA}.

We begin with statement (\ref{thm1}.1).  We consider only
finite square matrices whose entries are complex numbers
represented by quickly converging Cauchy sequences.  In
$\RCAo$, we say that a matrix $M$ \define{has a
  Jordan decomposition} if there are matrices $(U, J)$ such
that $M = U J U^{-1}$ and $J$ is a matrix consisting of
Jordan blocks. We call $J$ the \textit{Jordan canonical
  form} of $M$. The fundamental definitions and theorems
regarding the Jordan canonical form 
are presented by Halmos~\cite{Halmos-FDVS}*{Section~58}.
Careful formalization of (\ref{thm1}.1) shows that this principle can
be expressed by a $\Pi^1_2$ formula in $\Gamma_1$; the key point is
that the assumptions on $M$, $U$, $J$, and $U^{-1}$ can be
expressed using only equality of 
real numbers, which requires only universal
quantification.

\begin{lemma}\label{s3l1}
  $\RCAo$ proves that every $2 \times 2$ matrix has a
  Jordan decomposition.
\end{lemma}
\begin{proof}
  Let $M$ be a $2 \times 2$ matrix.  $\RCAo$ proves that the
  eigenvalues of $M$ exist and that for each eigenvalue
  there is an eigenvector.  (Compare Exercise~II.4.11
  of Simpson~\cite{Simpson-SOSOA}, which notes that the
  basics of linear algebra, including fundamental properties
  of Gaussian elimination, are provable in $\RCAo$.)  If the
  eigenvalues of $M$ are distinct, then the Jordan
  decomposition is trivial to compute from the eigenvalues
  and eigenvectors.  If there is a unique eigenvalue and
  there are two linearly independent eigenvectors then the
  Jordan decomposition is similarly trivial to compute.  

  Suppose that $M$ has a unique eigenvalue $\lambda$ but not
  two linearly independent eigenvectors.  Let $u$ be any
  eigenvector and let $\{u,v\}$ be a basis.  It follows that
  $(M - \lambda I)v = au + bv$ is nonzero.  Now $(M -
  \lambda I)(au + bv) = b(M-\lambda I)v$, because $u$ is an
  eigenvector of $M$ with eigenvalue $\lambda$.  This shows
  $(M - \lambda I)$ has eigenvalue~$b$, which can only
  happen if $b = 0$, that is, if $(M - \lambda I)v$ is a
  scalar multiple of $u$.  Thus $\{u,v\}$ is a chain of
  generalized eigenvectors of $M$; the Jordan decomposition can be
  computed directly from this chain.
\end{proof}

It is not difficult to see that the previous proof makes
use of the law of the excluded middle. 
%

\begin{remark}
  Proofs similar to that of Lemma~\ref{s3l1}
  can be used to show that for each standard natural number
  $n$ the principle that every $n \times n$ matrix has a
  Jordan decomposition is provable in $\RCAo$.  We do not
  know whether the principle that every finite matrix has a
  Jordan decomposition is provable in~$\RCAo$.
\end{remark}

The next lemma is foreshadowed by previous research.  It is well known
that the function that sends a matrix to its Jordan decomposition is
discontinuous.  Kohlenbach~\cite{Koh-HORM} has shown that, over the
extension $\RCAo^\omega$ of $\RCAo$ to all finite types, the existence
of a higher-type object encoding a non-sequentially-continuous
real-valued function implies the principle $\exists^2$. In turn,
$\rcaw + \exists^2$ proves every instance of the arithmetical
comprehension scheme.

\begin{lemma}
  The following principle implies arithmetical comprehension over\/ $\RCAo$
  \textup{(}and hence over $\rca$\textup{)}. For every
  sequence $\langle M_i \mid i \in \setN\rangle$ of $2
  \times 2$ real matrices, such that each matrix $M_i$ has
  only real eigenvalues, there are sequences $\langle U_i
  \mid i \in \setN \rangle$ and $\langle J_i \mid i \in
  \setN \rangle$ such that $(U_i,J_i)$ is a Jordan
  decomposition of $M_i$ for all $i \in \setN$.
\end{lemma}
\begin{proof}
  We first demonstrate a concrete example of the
  discontinuity of the Jordan form.  For any real $z$, let
  $M(z)$ denote the matrix
\[
M(z) = \begin{pmatrix}1 & 0 \\ 
z & 1 \end{pmatrix}.\]
The matrix $M(0)$ is the identity matrix, and so is its Jordan 
canonical form.
If $z \not = 0$ then $M(z)$ has the following Jordan decomposition:
\[
M(z) =
\begin{pmatrix}
1 & 0 \\
z & 1
\end{pmatrix}
=
\begin{pmatrix}
0 & 1\\
z & 0
\end{pmatrix}
\begin{pmatrix}
1 & 1 \\
0 & 1 
\end{pmatrix}
\begin{pmatrix}
0 & 1\\
z & 0
\end{pmatrix}
^{-1}.
\]
The crucial fact is that the entry in the upper-right-hand
corner of the Jordan canonical form of $M(z)$ is $0$ if $z =
0$ and $1$ if $z \not = 0$.

Let $h$ be an arbitrary function from $\setN$ to $\setN$.
We will assume the principle of the theorem and show that
the range of $h$ exists; this is sufficient to establish the desired result.
  It is well known that
$\mathsf{RCA}_0$ can construct a function $n \mapsto z_n$
that assigns each $n$ a quickly converging Cauchy sequence
$z_n$ such that, for all $n$, $z_n = 0$ if and only $n$ is
not in the range of $h$.  Form a sequence of matrices
$\langle M(z_n) \mid n \in \setN\rangle$; according to the
principle, there is an associated sequence of Jordan
canonical forms.  The upper-right-hand entry of each of
these canonical forms is either $0$ or $1$, and it is possible
to effectively decide between these two cases.  Thus, in
$\mathsf{RCA}_0$, we may form the range of $h$ using the
sequence of Jordan canonical forms as a parameter.
\end{proof}

We now turn to statement (\ref{thm1}.2).  Recall that the
standard formalization of the real numbers in $\RCAo$, as
described by Simpson~\cite{Simpson-SOSOA}, makes use of
quickly converging Cauchy sequences of rationals.
Alternative formalizations of the real numbers may be
considered, however.  We define a \textit{Dedekind cut} to
be a subset $Y$ of the rational numbers such that both $Y$ and $\setQ \setminus Y$ 
are nonempty, and if $p \in Y$ and $q < p$ then
$q \in Y$.  We say that a Dedekind cut $Y$ is
\textit{equivalent} to a quickly converging Cauchy sequence
$\langle a_i \mid i \in \setN\rangle$ if any only if the
equivalence
\[
q \in Y \Leftrightarrow q \leq \lim_{i\rightarrow \infty} a_i
\]
holds for every rational number $q$.   Formalization of (\ref{thm1}.2) shows that
it is in $\Gamma_1$.

Hirst~\cite{Hirst-RRRM} has proved the following results
that relate Cauchy sequences with Dedekind cuts.
Together with Theorem~\ref{719F}, these results show that statement
(\ref{thm1}.2) is provable in $\RCAo$ but not $\hawfi + \ac + \ipwef$.

\begin{lemma}[Hirst~\cite{Hirst-RRRM}*{Corollary 4}] The
  following is provable in $\RCAo$.  For each quickly
  converging Cauchy sequence $x$ there is an equivalent
  Dedekind cut.
\end{lemma}

\begin{lemma}[Hirst~\cite{Hirst-RRRM}*{Corollary~9}]
  The following principle is equivalent to weak K\"onig's lemma over
  $\RCAo$ \textup{(}and hence over $\rca$\textup{)}.
  For each sequence $\langle X_i \mid i \in \setN\rangle$ of
  quickly converging Cauchy sequences there is a sequence
  $\langle Y_i \mid i \in \setN\rangle$ of Dedekind cuts such that $X_i$ is
  equivalent to $Y_i$ for each $i \in \setN$.
\end{lemma}

Statement (\ref{thm1}.3), which is our final application of
Theorem~\ref{719F}, is related to countable posets.  In $\RCAo$, we
define a \textit{countable poset} to be a set $P \subseteq \setN$ with
a coded binary relation $\preceq$ that is reflexive, antisymmetric,
and transitive.  A function $f \colon \setN \rightarrow P$ is called
an \textit{enumerated filter} if for every $i,j \in \setN$ there is a
$k \in \setN$ such that $f(k) \preceq f(i)$ and $f(k) \preceq f(j)$,
and for every $q \in P$ if there is an $i \in \setN$ such that $f(i)
\preceq q$ then there is a $k \in \setN$ such that $f(k) = q$.  An
enumerated filter is called \textit{unbounded} if there is no $q \in
P$ such that $q \prec f(i)$ for all $i \in \setN$.  An enumerated
filter $f$ \textit{extends} a filter $g$ if the range of $g$ (viewed
as a function) is a subset of the range of~$f$.  If we modify the
usual definition of an enumerated filter to include an auxiliary
function $h\colon \setN^2 \to \setN$ such that for all $i$ and $j$,
$f(h(i,j))\preceq f(i)$ and $f(h(i,j))\preceq f(j)$, then
(\ref{thm1}.3) is in $\Gamma_1$.

Mummert has proved the following two lemmas about extending
filters to unbounded filters (see Lempp and Mummert~\cite{LM-FCP} and the
remarks after Lemma~4.1.1 of Mummert~\cite{Mummert-Thesis}).  These
lemmas show that (\ref{thm1}.3) is provable in $\RCAo$ but
not $\hawfi + \ac + \ipwef$.

\begin{lemma}[Lempp and Mummert~\cite{LM-FCP}*{Theorem~3.5}]
  $\RCAo$ proves that any enumerated filter on a countable
  poset can be extended to an unbounded enumerated filter.
\end{lemma}

\begin{lemma}[Lempp and Mummert~\cite{LM-FCP}*{Theorem~3.6}]
  The following statement is equivalent to arithmetical 
  comprehension over $\RCAo$ \textup{(}and
  hence over $\rca$\textup{)}. Given a sequence $\langle P_i \mid i
  \in \setN \rangle$ of countable posets and a sequence
  $\langle f_i \mid i \in \setN\rangle$ such that $f_i$ is
  an enumerated filter on $P_i$ for each $i \in \setN$,
  there is a sequence $\langle g_i \mid i \in \setN \rangle$
  such that, for each $i \in \setN$, $g_i$ is an unbounded
  enumerated filter on $P_i$ extending~$f_i$.
\end{lemma}

We close this section by noting that the proof-theoretic results of
section~\ref{sec3} are proved by finitistic methods.  Consequently,
constructivists might accept arguments like those presented here to
establish the non-provability of certain theorems from systems of
intuitionistic arithmetic.


\section{The {\it Dialectica} interpretation}\label{sec5}

In the proofs of section~\ref{sec3},
applications of G\"odel's {\it Dialectica} interpretation
can replace the applications of modified realizability.
One advantage of this substitution is that the constructive axiom
system can be expanded to include the scheme $\markov$, which formalizes
a restriction of the Markov principle. 

This gain has associated costs. First, the class of formulas for which
the uniformization results hold is restricted from $\Gamma_1$ to
the smaller class $\Gamma_2$ defined below.  Second, the independence
of premise principle $\ipwef$ is replaced with the weaker principle
$\ipwa$. Finally, the extensionality scheme~$\mathsf E$ is 
replaced with a weaker rule of inference
\[
{\sf {QF{\text -}{ER}}}\colon \text{From~}A_0 \to s=_\rho t\text{~deduce~}A_0 \to r[s/x^\rho]
 =_\tau r[t/x^\rho],
\]
where $A_0$ is quantifier free and $r[s/x^\rho]$ denotes the
result of replacing the variable $x$ of type $\rho$ by the
term $s$ of type $\rho$ in the term $r$ of type $\tau$.  
We denote the systems based on this rule of inference as $\whaw$ and $\whawfi$.

Extended discussions of G\"odel's \textit{Dialectica}
interpretation are given by Avigad and
Feferman~\cite{AF-HPT}, Kohlenbach~\cite{Koh-book}, and
Troelstra~\cite{troelstra73}.  The interpretation assigns to
each formula $A$ a formula $A^D$ of the form $\exists x
\forall y \,A_D$, where $A_D$ is quantifier free and each
quantifier may represent a block of quantifiers of the same
kind.  The blocks of quantifiers in $A^D$ may include
variables of any finite type.  

\begin{definition}
  We follow Avigad and Feferman~\cite{AF-HPT} in defining the
  \textit{Dialectica} interpretation inductively via the following six
  clauses, in which $A^D = \exists x \forall y \,A_D$ and $B^D =
  \exists u \forall v \,B_D$.
  \begin{list}{}{}
  \item [(1)] If $A$ a prime formula then $x$ and $y$ are both empty
    and $A^D = A_D = A$.
  \item [(2)] $(A \land B )^D = \exists x \exists u \forall y \forall
    v \,(A_D \land B_D)$.
  \item [(3)] $(A \lor B )^D = \exists z \exists x \exists u \forall y
    \forall v \,((z = 0 \land A_D) \lor (z=1 \land B_D))$.
  \item [(4)] $(\forall z \,A (z))^D = \exists X \forall z \forall y
    \,A_D (X(z) , y, z)$.
  \item [(5)] $(\exists z \,A (z))^D = \exists z \exists x \forall y
    \,A_D (x,y,z)$.
  \item [(6)] $(A \to B)^D = \exists U \exists Y \forall x \forall v
    \,(A_D (x, Y(x,v))\to B_D (U(x),v))$.
  \end{list}
  A negated formula $\neg A$ is treated as an abbreviation of $A \to
  \bot$.
\end{definition}

We begin our derivation of the uniformization results with a soundness
theorem of G\"odel that is analogous to Theorem~\ref{719B}.  A detailed
proof is given by Kohlenbach~\cite{Koh-book}*{Theorem 8.6}.

\begin{theorem}\label{722b}
Let $A$ be a formula in $\lang ( \whawfi )$.  If
\[
\whawfi + \ac + \ipwa + \markov \vdash \forall x\, \exists y A(x,y),
\]
then $\whawfi  \vdash \forall x A_D (x,t(x))$, where $t$ is a suitable term of $\whawfi$.
\end{theorem}

To prove our uniformization result, we will need to convert $A^D$ back
to~$A$.  Unfortunately, $\rcaw$ can only prove $A^D \to A$ for certain
formulas.  The class~$\Gamma_2$, as found in (for example) Definition
8.10 of Kohlenbach~\cite{Koh-book}, is a subset of these formulas.

\begin{definition}
$\Gamma_2$ is the collection of formulas in $\lang ( \whawfi )$
 defined inductively as follows.
\begin{list}{}{}
\item [(1)]  All prime formulas are elements of $\Gamma_2$.
\item [(2)] If $A$ and $B$ are in $\Gamma_2$, then so are $A \land B$,
  $A\lor B$, $\forall x A$, and $\exists x A$.
\item [(3)] If $A$ is purely universal and $B \in \Gamma_2$, then
  $(\exists x A \to B) \in \Gamma_2$, where $\exists x$ may represent
  a block of existential quantifiers.
\end{list}
\end{definition}

Kohlenbach~\cite{Koh-book}*{Lemma 8.11} states the following result for
$\whawfi$.  Since $\rcaw$ is an extension of $\whawfi$, this suffices for the proof
of the uniformization result, where it acts as an analog of Lemma~\ref{719D}.

\begin{lemma}\label{722d}
Let $A$ be a formula of $\lang(\whawfi)$ in $\Gamma_2$.  Then\/ $\whawfi \vdash A^D \to A$.
This result also holds for $\whaw$ for formulas in $\lang(\whaw)$.
\end{lemma}

\begin{proof}
The proof is carried out by an external induction on formula complexity with
cases based on the clauses in the definition of $\Gamma_2$.  For details,
see the proof of part~(iii) of~Lemma~3.6.5 in Troelstra \cite{troelstra73}.
The proof of each clause depends only on the definition of the \textit{Dialectica}
interpretation and intuitionistic predicate calculus.  Consequently, the same
argument can be carried out in $\whaw$.
\end{proof}


We can adapt our proof of Lemma~\ref{719E} to obtain the following
term extraction result.

\begin{lemma}\label{722e}
Let\/ $\forall x^\rho \exists y^\tau A(x,y)$ be a sentence 
of $\lang(\whawfi)$
in $\Gamma_2$ with arbitrary types
$\rho$ and~$\tau$.  If\/
$
\whawfi  + \ac + \ipwa + \markov \vdash \forall x^\rho \exists y^\tau A(x,y),
$
then $\rcaw \vdash \forall x^\rho A(x, t(x))$, where $t$
is a suitable term of $\whawfi$.
\end{lemma}

Substituting Lemma \ref{722e} for the use of Lemma \ref{719E} in
the proof of Theorem~\ref{719F}, we obtain a proof of the
{\it Dialectica} version of our uniformization result.

\begin{theorem}\label{722f}
Let\/ $\forall x \exists y A(x,y)$ be a sentence 
of $\lang(\whawfi)$ in $\Gamma_2$.  If
\[
\whawfi  + \ac + \ipwa + \markov \vdash \forall x \exists y\, A(x,y),
\]
then
\[
\rcaw \vdash \forall \seqx \, \exists \seqy \forall n\, A(x_n, y_n).
\] Furthermore, if $x$ and $y$ are both type $1$
\textup{(}set\textup{)} variables, and $\forall x \exists y
A(x,y)$ is in $\lang ( \rca )$, then\/ $\rcaw$ may be
replaced by\/ $\rca$ in the implication.
\end{theorem}

As was the case in section~\ref{sec3}, these results can
be recast in settings with restricted induction. As noted by
Kohlenbach \cite{Koh-book }*{section 8.3}, Theorem \ref{722b}
also holds with $\whawfi$ replaced by $\whaw$.  Applying the
restricted-induction version of Lemma \ref{722d} leads to the
restricted form of Lemma \ref{722e}.  Combining this with
the conservation result for $\rcaw _0$ over $\RCAo$ (Theorem~\ref{consrcao}) leads to
a proof of the following version of Theorem \ref{722f}.

\begin{theorem}\label{restrdialectica}
Let\/ $\forall x \exists y A(x,y)$ be a sentence of $\lang(\whaw)$ 
in $\Gamma_2$.  If
\[
\whaw  + \ac + \ipwa + \markov \vdash \forall x\, \exists y\, A(x,y),
\]
then
\[
\rcaw_0  \vdash \forall \seqx \, \exists \seqy \forall n\, A(x_n, y_n).
\] Furthermore, if $x$ and $y$ are both type $1$
\textup{(}set\textup{)} variables, and $\forall x \exists y
A(x,y)$ is in $\lang ( \RCAo )$, then\/ $\rcaw_0$ may be
replaced by\/ $\RCAo$ in the implication.
\end{theorem}



Uniformization results obtained by the {\it Dialectica} interpretation
are less broadly applicable than those obtained by modified
realizability, due to the fact that $\Gamma_2$ is a proper subset of
$\Gamma_1$.  In practice, however, the restriction to $\Gamma_2$ may
not be such a serious impediment.  Examination of the statements in
Theorem \ref{thm1} shows that the hypotheses in their implications are
purely universal, and consequently each of the statements is in
$\Gamma_2$.  Thus an application of Theorem~\ref{722f} shows that
Theorem~\ref{thm1} holds with $\hawfi + \ac + \ipwef$ replaced by
$\whawfi + \ac + \ipwa + \markov$.

While $\Gamma_2$ may not be the largest class of formulas for which an
analog of Theorem~\ref{restrdialectica} can be obtained, any class
substituted for $\Gamma_2$ must omit a substantial collection of
formulas.  For example, imitating the proof of
Kohlenbach~\cite{Koh-goodman}, working in $\whaw + \ac$ one can deduce
the $\Pi^0_n$ collection schemes, also known as ${\sf B} \Pi^0_n$.
These schemes contain formulas that are not provable in $\RCAo$, 
and any class of
formulas for which Theorem \ref{restrdialectica} holds must omit such 
formulas. The same observation holds for Theorem \ref{719J}.


\bibliographystyle{asl}

\end{document}